\documentclass[ a4paper
              , 10pt
              , reqno
              , submission
              ]{eptcs}
\providecommand{\event}{ACT2022}
\title{Fibrational linguistics (\fiblang): First concepts}

\makeatletter%
\@ifclassloaded{amsart}{
  \usepackage{fouche/fouche}
\author{Fabrizio \textsc{Genovese}}
\author{Fosco \textsc{Loregian}}
\author{Caterina \textsc{Puca}}
\address{ Fabrizio Genovese\newline
	Statebox,\newline %
	\url{fabrizio.romano.genovese@gmail.com}%
}
\address{ Fosco Loregian\newline
	Tallinn University of Technology,\newline %
	\url{fosco.loregian@taltech.ee}%
}
\address{ Caterina Puca\newline
	Cambridge Quantum Computing,\newline %
	\url{caterpuca@gmail.com}%
}
\usepackage{hyperref}
\hypersetup{ hyperfootnotes=true
	, bookmarksnumbered=true
	, bookmarksopen=true
	, breaklinks=true
	, urlcolor=red!40!black
	, colorlinks%
	, urlcolor=RoyalBlue
	, linkcolor=Red
	, citecolor=Green
}
\usepackage{tipa}
\newcommand{\elts}[2]{
  \mathchoice{#1 \rotatebox[origin=c]{15}{$\int$} #2}
  {#1 \rotatebox[origin=c]{15}{$\int$} #2}
  {#1 \rotatebox[origin=c]{15}{\scriptsize$\int$} #2}
  {#1 \rotatebox[origin=c]{15}{\tiny$\int$} #2}
}
\thanks{The first author was supported by the \href{https://gitcoin.co/grants/1086/independent-ethvestigator-program}{Independent Ethvestigator Program}. The second author was supported by the ESF funded Estonian IT Academy research measure (project 2014--2020.4.05.19--0001). A special thanks goes to Bettino, the fat cat of figures 1, 2, and 3, and Peppa the black evil feline, for having provided valuable suggestions and
comments after an attentive proofreading of an early version of this manuscript.}
}%
{
  \usepackage{hyperref}
\hypersetup{ 
	, bookmarksnumbered=true
	, bookmarksopen=true
	, breaklinks=true
	, urlcolor=blue!40!black
	, colorlinks%
	, urlcolor=RoyalBlue
	, linkcolor=blue!40!black
	, citecolor=Green
}
\usepackage{fouche/act-fouche}
\author{ Fabrizio \textsc{Genovese}
	\email{\orcidNum{0000-0001-7792-1375}}
	\institute{Statebox}%
	  \thanks{Supported by the \href{https://gitcoin.co/grants/1086/independent-ethvestigator-program}{Independent Ethvestigator Program}.}
	\email{fabrizio.romano.genovese@gmail.com}
	\and
	Fosco \textsc{Loregian}
	\email{\orcidNum{0000-0003-3052-465X}}
	\institute{Tallinn University of Technology}%
	  \thanks{Supported by the ESF funded Estonian IT Academy research measure (project 2014--2020.4.05.19--0001).}
	\email{fosco.loregian@gmail.com}
	\and
	Caterina \textsc{Puca}
	\institute{Sapienza University of Rome}
	\email{caterpuca@gmail.com}
}
\def\authorrunning{ Genovese%
	                , Loregian%
	                , Puca%
                  }
\def\titlerunning{Fibrational linguistics}
\usepackage{ newtxtext
	         , newtxmath
					 }

\newcommand{\elts}[2]{
  \mathchoice{\rotatebox[origin=c]{15}{$\int$} #2}
  {\rotatebox[origin=c]{15}{$\int$} #2}
  {\rotatebox[origin=c]{15}{\scriptsize$\int$} #2}
  {\rotatebox[origin=c]{15}{\tiny$\int$} #2}
}
}%
\makeatother

\usepackage{doi}
\usepackage[ backend=bibtex
	         , natbib=true
	         , style=numeric
	         , sorting=nyt
	         , maxbibnames=10
	         , giveninits=true
	         , backref=true
	         , url=false
	         , doi=true
	         , isbn=false
	         , arxiv=abs
	         , date=year
	         , abbreviate = true
           ]{biblatex}

\AtBeginBibliography{\small}

\def\Spk{\cate{Fab}}

\newcommand{\doublear}[2]{\draw[->, shorten >=1mm, shorten <=1mm] #1 to[bend left=15] #2;
\draw[->, shorten >=1mm, shorten <=1mm] #2 to[bend left=15] #1;
}

\def\lCell{\xymatrix{
\ar@{>->}[d]_{e_f} \ar@{}[dr]|{\alpha_L} \ar[r]|*@{|}^p& \ar@{>->}[d]^{e_g}\\
\ar[r]|*@{|}_q&
}}
\def\rCell{\xymatrix{
\ar@{->>}[d]_{m_f} \ar@{}[dr]|{\alpha_R} \ar[r]|*@{|}^p& \ar@{->>}[d]^{m_g}\\
\ar[r]|*@{|}_q&
}}
\def\Cell{\xymatrix{
\ar[d]_f \ar@{}[dr]|\alpha \ar[r]|*@{|}^p& \ar[d]^g\\
\ar[r]|*@{|}_q&
}}

\usepackage{xspace}
\usepackage{graphbox}

\usepackage[russian, english]{babel}

\newcommand{\orcidLogo}{\includegraphics[align=c]{orcid_logo.pdf}\kern.3em}
\newcommand{\orcidNum}[1]{\href{https://orcid.org/#1}{#1}}

\NewDocumentCommand{\fib}{O{r}}{\ar@{->>}[#1]|*@{*}}
\NewDocumentCommand{\opfib}{O{r}}{\ar@{->>}[#1]|*@{o}}

\usepackage{cleveref}

\def\fiblang{FibLang\@\xspace}

\newcommand{\framew}[1]{{{#1}^\sharp}}

\usepackage{dialogue,csquotes}
\usepackage[T1,T2A]{fontenc}
\usepackage{anyfontsize}
\usepackage{tempora}
\usepackage[super]{nth}
\begin{document}

\maketitle
\begin{abstract}
    We define a general mathematical framework for linguistics based on the theory of fibrations,
    called \fiblang.

    We start by modelling the interaction between linguistics and cognition in the most general way possible,
    focusing on conceptually motivating any assumption we make. The advantage is that \fiblang remains agnostic to
    any particular axiomatization of grammar one may choose. As such, it is compatible with
    existing categorical language models (such as DisCoCat), providing a formally sound framework to apply
    mathematical tools developed in the context of category theory, mainly categorical logic, to the study of language.
\end{abstract}

\section{Introduction}
The work of N. Chomsky is considered revolutionary in linguistics. Among the many reasons this approach is impactful, the major one is that it gave researchers a formal system to reason about linguistic phenomena~\cite{Chomsky1956}. Later on, this work was recast in category-theoretic terms by J. Lambek, first employing sequent calculus \cite{lambek1958mathematics} and ultimately the theory of so-called \emph{pregroups} \cite{093fd03b7dc604c64124630f20a8f01231253397,lambek2010compact,lambek2004bicategories}, particular instances of compact closed categories.

The work of Lambek conjugates grammar in algebraic terms: it postulates that linguistic structures are made up of linguistic `atoms' -- \emph{nouns}, \emph{adjectives}, \emph{verbs} -- that can be put together and then \emph{reduced} into complex linguistic structures -- \emph{sentences}, \emph{noun phrases} -- leveraging the pregroup structure. As such, language can be treated as a compositional structure because it results from a composition of its parts.

In the 2000s a connection between the pregroup approach and \emph{categorical quantum mechanics} \cite{Abramsky2004, Coecke2017} started to be explored, providing a useful definition of semantics in distributional terms, commonly known as \emph{DisCoCat}~\cite{228d9e4b69926594fd26080f4cfaa9ecfca44eb3}.
Here, grammatical types are mapped to vector spaces containing statistical information about the words of that type. Grammatical reductions are mapped to linear operations, fusing the statistical properties of the linguistic constituents into statistical properties of the resulting sentences.
This has been particularly useful in quantum natural language processing and spawned even more research, culminating in recent contributions such as \emph{DisCoCirc}~\cite{coecke2019mathematics}, that are currently being tested and implemented on quantum computers~\cite{2001.00862,meichanetzidis2020quantum}.

In parallel to this, categories have been used to describe \emph{models of meaning}, where the focus shifts from computational applications to \emph{cognitive} ones. Multiple models have been used as a starting point for categorical formalization, such as \emph{conceptual spaces}, originally defined by G\"ardenfors~\cite{Alfred-2007} and formalized categorically in~\cite{a4051b40c62acb13ce5a9e3b01d3971d653ac105}.

The conceptual spaces approach has been fruitful but did not come without shortcomings, such as the difficulty of producing something more than toy models~\cite{Gogioso2016}. Moreover, it gave for granted many assumptions that may not be so readily accepted by some researchers~\cite{gayral:hal-00084948}.
\subsection*{Scope of the paper}
The present work attempts to define a general theory of language and meaning, which we refer to as `fibrational linguistics' (\fiblang). We have two main `desiderata' in mind:
\begin{itemize}
 \item\emph{We are agnostic on the structure of language}. Following Lambek, we model language as a category $\clL$, but we strive to make no additional assumption about $\clL$ and any additional properties that it might satisfy. This highly general attitude informs our \autoref{ling_doctrina}, and, more generally, all the `foundational' parts of this work in §\ref{ct_for_real}.
 \item \emph{We treat semantics as a black box}. The spate of applications of Lambek's approach of category theory to linguistics relies on the fact that we can model language as a syntactic object; works such as~\cite{a4051b40c62acb13ce5a9e3b01d3971d653ac105} highlight the semantic nature of meaning. Systems like \cite{Gardenfors2000, Gardenfors2014} build on notions of \emph{convexity} that work well in modelling specific concepts such as colour and taste but fall short when one attempts to use them in full generality.

 We concur in positing that the interplay between language and meaning can be captured through category-theoretic methods but try to do so by giving up on any controversial assumption.
\end{itemize}
To attain these two ambitious goals, we will turn some postulates about how categorial linguistics is done, and in particular the r\^ole of \emph{meaning}, upside-down. This can be summarized in the slogan `grammar is algebraic, meaning is coalgebraic' that we will come to motivate in the following sections.

To make such a general approach workable, we will employ some basic facts in the theory of \emph{fibrations} of categories. We assume some familiarity with category theory (see, e.g. \cite{working-categories, KS2}) and no familiarity with fibrations whatsoever, for which the go-to reference for the subject is~\cite{CLTT}.

We will intentionally employ friendly language, especially when building examples and presenting simplified versions of the main definitions. This design choice has two main motivations: first, our main concern is to convey the intuitive concepts that justify our stance on the philosophy of language without making things too heavy on the formal side. For this reason, we also decided to postpone the treatment of possible applications to another paper~\cite{fiblang1}. Second, we hope this work might catch the attention of members of communities that usually lack a solid mathematical background --for example, actual linguists.
\section{What meaning can(not) be}\label{sec:whatmeaningcannotbe}

A genuine question that spawns in the mind of anyone curious about language is the following: `what goes on in people's heads while they speak? What is \emph{meaning}?'
Here, experience teaches a hard lesson: almost any answer will inevitably leave much to be desired. The reasons for this are many, but we humbly pinpoint a couple:
\begin{description}
  \item[Different people think in different ways:] some people, for instance, have a distinctly spatial intuition and visualise the object of their discourse; others do not, a condition known as \emph{aphantasia} \cite{zeman2015lives, crowder2018differences}. Some people have an internal monologue when they read \cite{perrone2014little}; others do not \cite{langland2015inner}. As such, what could be very intuitive for someone could feel wrong for someone else, and while there might be general patterns to describe language acquisition, formation and usage, no model for \emph{cognition} is universal. Indeed, according to the Russian School, cognition is strongly dependent on the social context in which it occurs \cite{lave1991situated}.
  \item[There is more than language going on:] in modelling meaning, we are using language to formalise and convey what we think meaning is. This tacitly assumes that cognitive processes can be fully described within language. Thinking about it, this constitutes a huge, possibly unjustified leap of faith, or at least an extreme simplification because some cognitive processes defining abstract terms are local and dependent on the culture of reference \cite{william1921elements} 
\end{description}
From the perspective of someone interested in modelling meaning mathematically, these considerations make the `problem of meaning' almost untreatable: the only thing we may safely venture to say is that a putative category $\clD$ of `meanings' shall be treated as some black box attached to a speaker $p$ of a language $\clL$. The internal structure of $\clD$ is modelled on $\clL$, but its underpinnings are mostly inaccessible to anyone but $p$.

This means that, at least at a superficial level, there is little to no mathematics involved in describing what exactly happens while one \emph{learns} something, or what exactly happens when ideas inform language, and vice versa. However, even when assuming that a theory of meaning is essentially inaccessible, we realise that there is something mathematical to say about what happens when language is shared, extended, or changed upon use because one can attempt to model these circumstances as the effects of operations performed on the category $\clL$ that encodes the language in study.

As such, our first definition attempts to encode the fact that given a language $\clL$, its speakers `exist' in a universe parametric over $\clL$; if language is a mathematical object (a category), a speaker of said language is \emph{another} object modelled over that category.
\begin{heuristics}\label{why_functors}
The following intuitive idea informs the definition: a speaker $p$ `gathers' meanings that they attribute to a word $L\in \clL$ in a collection $\clD_L$ forming a `bundle' over $\clL$. The exact way $\clD_L$ is formed coincides with some process of `understanding' that a speaker $p$ undergoes to clarify the meaning of $L$. Such a process is unavoidably contextual: the collection $\clD_L$ is big or small depending on the `environment' in which $p$ is immersed. The exact nature of this formation process is of great philosophical interest, but at the moment it does not concern us very much: the focus is on the conglomerate $\clD$ formed as the `union' $\bigcup_L\clD_L$ and on the fact that albeit in this picture $\clL$ remains a syntactic object,
\begin{itemize}
\item its nature is never specified, and thus can be `anything' we deem natural to describe the language $\clL$ from outside, and
\item $\clD$ stands on a different ground than $\clL$: if the latter contains words, the former contains \emph{meaning}.
\end{itemize}
\end{heuristics}
This leads us to formulate a very concise definition:
\begin{definition}[Speaker]\label{def:speaker}
  A \emph{speaker} is a functor $p : \clD^p \to \clL$.
\end{definition}
This might seem just a pointless relabeling of a mathematical object into an evocative concept. For this reason, we now aim at motivating it better, culminating later on with a series of examples that will hopefully clarify in what sense the parallel between languages and categories can be re-interpreted, departing from a purely syntactic perspective to one where syntax and semantics determine, at least partly, each other.

In \autoref{def:speaker}, $\clD^p$ represents a collection of \emph{concepts}, or \emph{meanings}, that a particular speaker $p$ attributes to words and wants to convey at a given time in a given context. It can be thought of as a snapshot of the speaker's brain and, as such, it can (and will) change under `pressure' (time, social interactions, changes in grammar and semantics due to cultural shifts, wars, floods, technological advancement\dots).

The category $\clL$ instead is meant to represents language; at this stage of the discussion it is not clear what kind of properties we have to require for $\clL$: it can be treated as a purely syntactic object (borrowing from the `tradition' of categorical linguistics, from \cite{lambek1958mathematics} to \cite{228d9e4b69926594fd26080f4cfaa9ecfca44eb3,d39a3564921e3dfa0ecea92cddc2f0cb7e611511}) or not: we argue that following the latter path leads to a more fruitful interpretation for how category theory can shed light on linguistics, and we further develop our posture in~\cite{fiblang1}.
\begin{remark}\label{rem: discocat is upside down}
  Readers familiar with approaches such as DisCoCat may be surprised by~\autoref{def:speaker}, as previous work assumes functors going from grammar to semantics and not the other way around. Our approach is motivated by the fact that the grammar category, as a purely syntactical structure, is usually `thinner' than the semantics category. For instance, there may be many different meanings mapping to the same word or sentence and many more morphisms between meanings than grammatical reductions. In the case of DisCoCat, this handwavy consideration can be made precise in that a pregroup $\clL$ is formally a thin category, and this has drawbacks: indeed, considering functors from a pregroup to, say, the category of finite-dimensional vector spaces gives rise to pathological behaviour as stressed out in the work of Preller \cite{Preller2007,Preller2014}, further justifying \autoref{def:speaker}, that moreover reflects the idea of grammar as an abstraction projected out from meaning and language.
\end{remark}

The internal structure of $\clL$, intended as its properties as a category, can be a subject of endless debate, given that there is no widespread, definitive consensus about what language is --apart from a smoky definition ultimately motivated by a philosophical stance.

Given this, in \fiblang we posit $\clL$ being at least a category, mainly since it has been widely accepted --starting at least from the earliest works of Chomsky-- that language has some degree of \emph{compositionality}. But we make no further assumption on the $\clL$ that can be interesting, and instead, we formulate an extremely loose notion of \emph{linguistic doctrine} in \autoref{ling_doctrina}: a \emph{(linguistic) doctrine} a subcategory $\cate{LD}\subseteq\Cat$ whose objects are categories defined by some additional property or structure
(see also the rest of \autoref{ct_for_real}, and a roundup of examples in \autoref{roundup_1}, \ref{roundup_2}, and \ref{roundup_3}). Usual classes of posetal categories \cite{lambek1997type,ajdukiewicz35,lambek1958mathematics,grishin1983generalization} are legitimate choices for linguistic doctrines, but other choices are possible.
\begin{example}\label{roundup_1}
  The categories of Lambek protogroups and pregroups \cite{lambek1997type}, Ajdukiewicz-Bar-Hillel pomonoids \cite{ajdukiewicz35,barhillel1953quasi}, residuated pomonoids \cite{lambek1958mathematics}, Grishin pomonoids \cite{grishin1983generalization} are all examples of linguistic doctrines in the sense of \autoref{ling_doctrina}.
\end{example}
\begin{example}\label{roundup_2}
  Likewise, `syntactic categories' of various sorts, like categories with a prescribed class of limits, are all examples of linguistic doctrines.
\end{example}
If language is (with few exceptions) widely accepted to be compositional to some degree, on the other hand, meaning (i.e., the `mental image'\footnote{For illustrative purposes only, from now on we will assume that concepts are mental images.} that $p$ builds upon $\clL$) stands on a different ground: it certainly is `something that $p$ does with language'; it certainly is subject to a certain degree of compositionality --from one extreme, not compositional at all, in which case $\clD^p$ would be a \emph{discrete} category, just a set, to the opposite. But, most importantly, it can be compositional in \emph{different} ways depending on different speakers in different social or learning contexts.
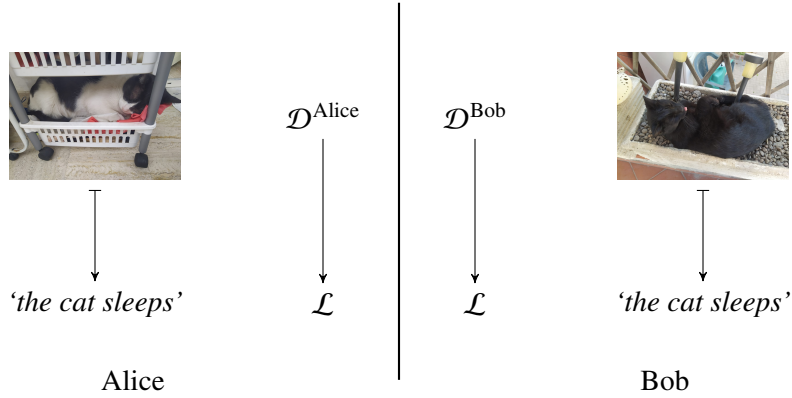
\begin{figure}[htp]
  \begin{center}
    \begin{tikzpicture}[>=stealth']
      \node (bettinoDorme) at (0,3.5) {\includegraphics[scale=0.04]{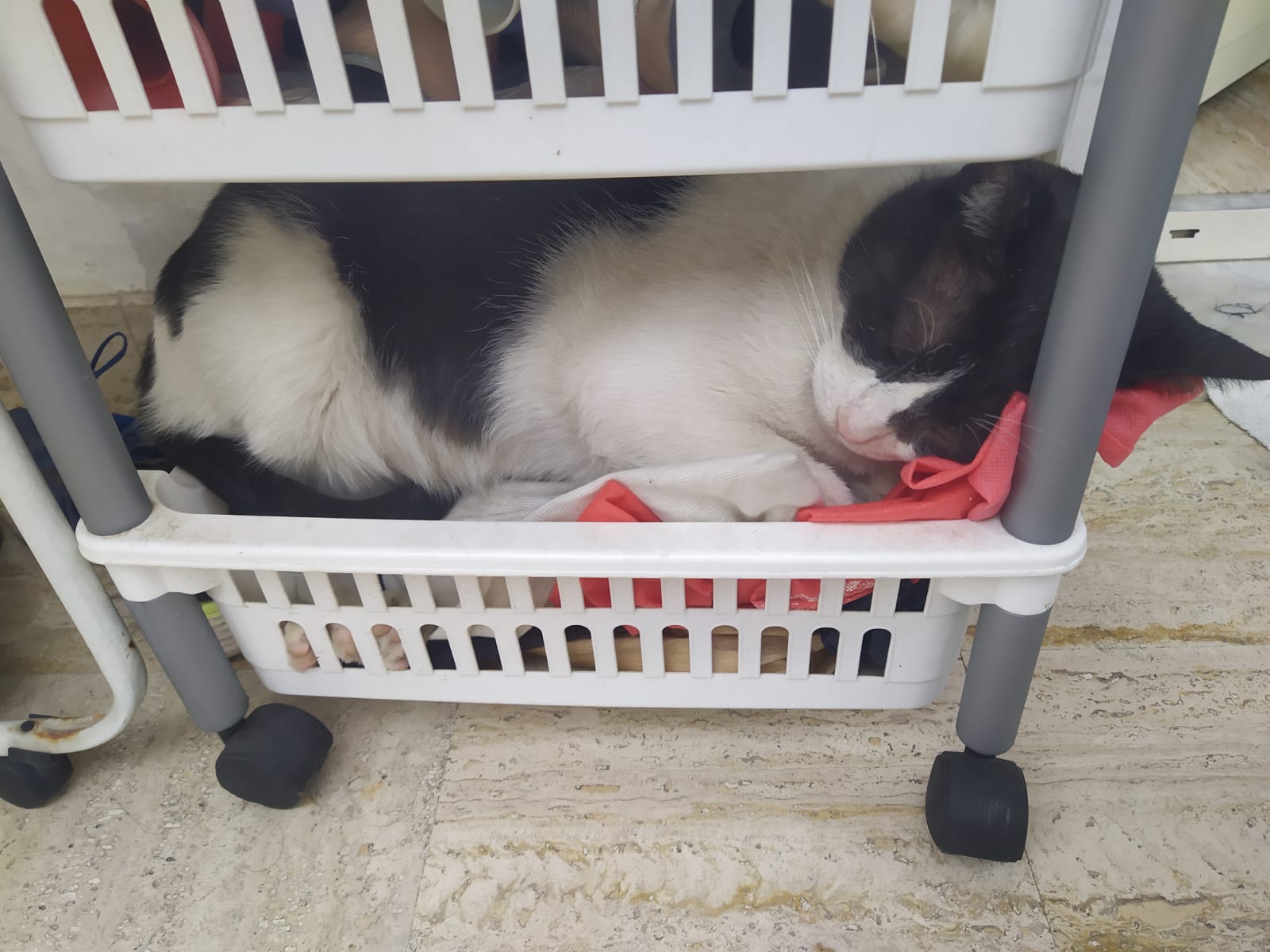}};
      \node (D) at (3,3.5) {$\clD^{\text{Alice}}$};
      \node (D') at (5,3.5) {$\clD^{\text{Bob}}$};
      \node (peppaDorme) at (8,3.5)  {\includegraphics[scale=0.04]{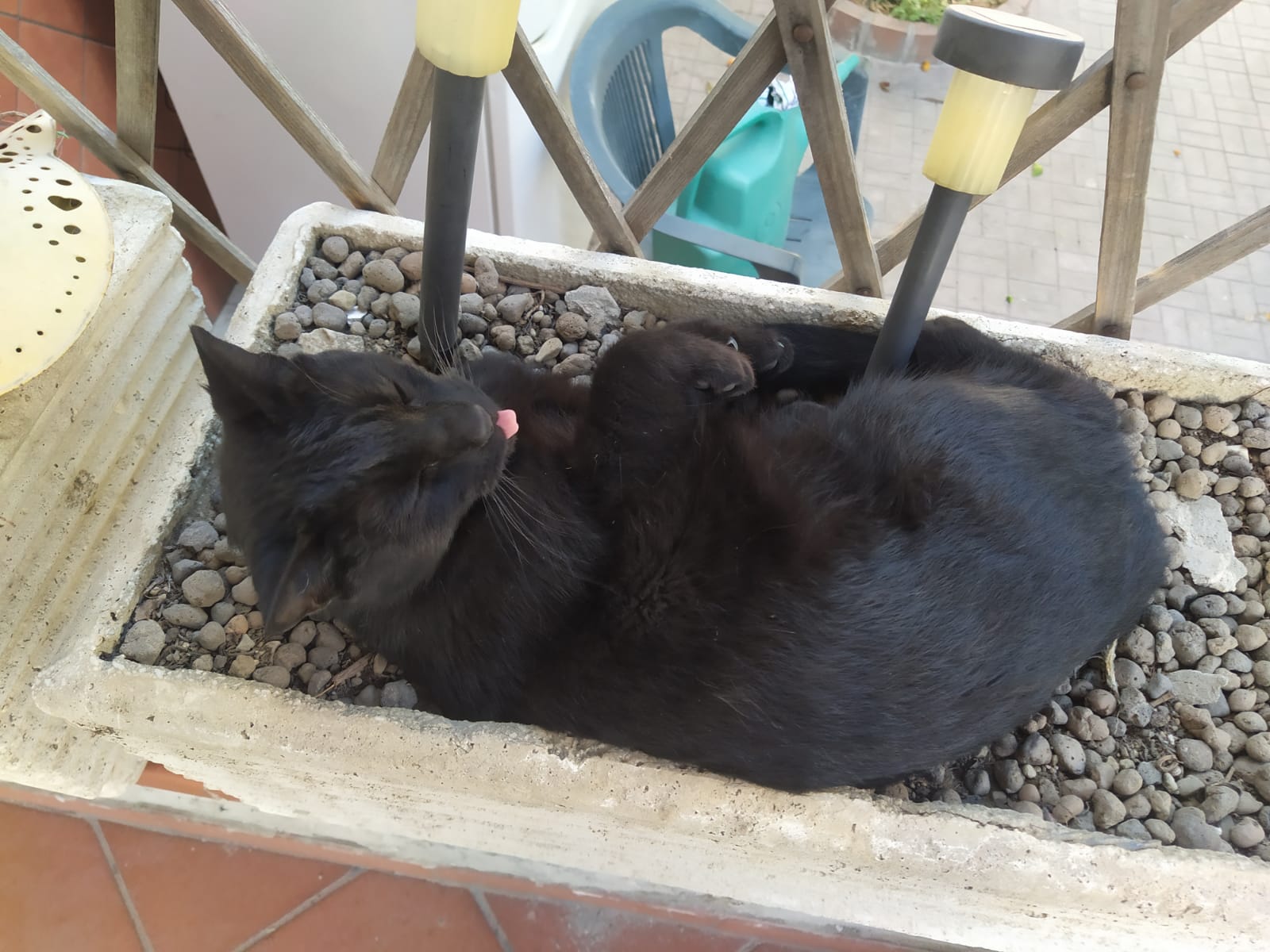}};

      \node (sentence1) at (0,1) {\emph{`the cat sleeps'}};
      \node (language1) at (3,1) {$\clL$};
      \node (language2) at (5,1) {$\clL$};
      \node (sentence2) at (8,1) {\emph{`the cat sleeps'}};

      \node (A) at (0.5,0) {Alice};
      \node (B) at (7.5,0) {Bob};

      \draw[thick] (4,0) -- (4,5);
      \draw[|->] (bettinoDorme) -- (sentence1);
      \draw[->] (D) -- (language1);
      \draw[->] (D') -- (language2);
      \draw[|->] (peppaDorme) -- (sentence2);
    \end{tikzpicture}
  \end{center}
  \caption{Different speakers may imagine things differently.}\label{fig: the cat sleeps}
\end{figure}

\begin{example}
  Two speakers, Alice and Bob, may interpret the sentence \emph{`the cat sleeps'} in different ways, as in~\autoref{fig: the cat sleeps}. Alternatively, \autoref{fig: the cat sleeps} may represent the meaning of the same sentence at different times for the same speaker.

  Even worse, the same speaker $p$ may think about different cats in the sentences \emph{`the cat sleeps'} and \emph{`the cat is fat'}. In this case, it is debatable --at least since the times of Plato-- that there is a single meaning of the word \emph{cat} that can be used in both sentences to infer the meaning of the sentences as a whole.
\end{example}





%
%

\begin{remark}[A counterpoint to \autoref{why_functors}]
  A mere functor $\clD^p\to \clL$ does not contain enough information in order for the categories $\clD_L$ to vary `coherently' with the structure of $\clL$; this gives us a way to introduce the main mathematical tool in \fiblang, the theory of \emph{fibrations} over a category.

  Intuitively speaking, this is exactly what a fibration $\var[\framew{p}]{\clE}{\clL}$ is for: define a category $\clE$ `over' $\clL$ in such a way $\clE_L$ varies smoothly as a subcategory of $\clE$, as long as $L$ varies in $\clL$. This translates into the slightly more formal request that the assignment $L\mapsto \clE_L$ is a functor, cf. \autoref{def:_fibration}, \autoref{fig_fibres}  below and \autoref{groco} in the appendix.
\end{remark}

\section{Fibrations over languages}\label{sec: fibrations over languages}
Intuitively, a fibration is a functor $\framew{p}: \clE \to \clC$ that realises the category $\clE$ as a `covering' of $\clC$, in such a way that each morphism $f : C\to C'$ in $\clC$ can be lifted to a morphism in $\clE$, so (in the case of what we call a \emph{discrete} fibration) to induce functions between the fibres in $\clE$. The original definition of this class of functors was given by A. Grothendieck, mimicking a similar concept in algebraic topology, where a fibration (of topological spaces) is defined as a continuous function $\var[p]{E}{B}$ lifting paths $\gamma : [0,1] \to B$ to paths in $E$. The analogy is meaningful, categories being in suitable sense models of `directed' spaces (cf. \cite{book91393877}), but the categorical theory of fibrations exhibits an additional layer of complexity since we must distinguish fibrations (inducing morphisms in the opposite direction of $f : C \to C'$) and \emph{op}fibrations, inducing morphism in the same direction of $f$.
\begin{definition}[Fibration]\label{def:_fibration}
  A functor $\framew{p} : \clE \to \clC$ is a \emph{discrete fibration} (for us, just a \emph{fibration}) if, for every object $E$ in $\clE$ and every morphism $f: C \to \framew{p}E$ with $C$ in $\clC$, there exists a unique morphism $h: E' \to E$ such that $\framew{p}h = f$.
\end{definition}
For notational convenience, functors that are fibrations will always be denoted with a $\sharp$. Basic facts on fibrations are recalled in a separate Appendix.
\begin{notation}\label{notat_for_fib}
  The domain of a fibration $\framew{p} : \clE \to \clC$ is usually called the \emph{total category} of the fibration, and its codomain is the \emph{base category}. Given any functor $p$ we can define the \emph{fibre} of $p$ over an object $C\in\clC$, i.e. the subcategory $\clE_C = \{f : E\to E'\mid p f = 1_C\}\subseteq \clE$.
\end{notation}
\begin{remark}
  The fibration property for $p : \clE\to\clC$ entails that fibres are discrete subcategories of the total category (that is, they are sets) and given a morphism $f : C \to C'$ in $\clC$ we can define a function $f^*$, called \emph{reindexing function}, from the fibre $\clE_{C'}$ to the fibre $\clE_C$, sending $X\in \clE_{C'}$ to the domain $X'$ of the unique $h : X' \to X$ such that $\framew{p}h = f$, which is by construction an object in $\clE_C$.
\end{remark}
We have represented an elementary example of fibration in the following figure, where the grey rectangles are the \emph{fibres} of the fibration.
\begin{figure}[ht]
\begin{center}
  \scalebox{0.8}{
    \begin{tikzpicture}[node distance=1.3cm,>=stealth',bend angle=45,auto]
      \node (C) at (-3,0) {$\clC$};
      \node (E) at (-3,2.3) {$\clE$};

      \node (1a) at (0,0) {$A$};
      \node (3a) at (3,0)  {$B$};
      \node (5a) at (6,0) {$C$};

      \draw[->] (1a) to node[above] {$f$} (3a);
      \draw[->] (3a) to node[above] {$g$} (5a);

      \node (EA) at (0.7,1.3) {$\clE_A$};
      \draw[thick, draw=black!50, fill=black!10] (-0.5,1.5) rectangle (0.5,3);
      \node[font=\scriptsize] (a0) at (-.25,2.8) {$A_0$};
      \node[font=\scriptsize] (a1) at (0.25,2.3) {$A_1$};
      \node[font=\scriptsize] (a2) at (-.25,1.7) {$A_2$};

      \node (EB) at (3.7,1.3) {$\clE_B$};
      \draw[thick, draw=black!50, fill=black!10] (2.5,1.5) rectangle (3.5,3);
      \node[font=\scriptsize] (b0) at (3.25,2.8) {$B_0$};
      \node[font=\scriptsize] (b1) at (2.75,2.3) {$B_1$};
      \node[font=\scriptsize] (b2) at (3.25,1.7) {$B_2$};

      \node (EC) at (6.7,1.3) {$\clE_C$};
      \draw[thick, draw=black!50, fill=black!10] (5.5,1.5) rectangle (6.5,3);
      \node[font=\scriptsize] (c0) at (6,2.8) {$C_0$};
      \node[font=\scriptsize] (c1) at (6,1.7) {$C_1$};

      \draw[->] (E) to node[left] {$\framew{p}$} (C);

      \draw[<-|, out=0, in=180] (a0) to (b0);
      \draw[<-|, out=0, in=180] (a2) to (b1);
      \draw[<-|, out=0, in=180] (a2) to (b2);

      \draw[<-|, out=0, in=180] (b0) to (c1);
      \draw[<-|, out=0, in=180] (b2) to (c0);

      \draw[dotted, -]  (1a.north) -- (0,1.25);
      \draw[dotted, -]  (3a.north) -- (3,1.25);
      \draw[dotted, -]  (5a.north) -- (6,1.25);
    \end{tikzpicture}}
\end{center}
\caption{Depiction of a fibration. The arrows between elements of the fibres diplay the actions of the reindexing functions. The fibres of $\framew{p}$ over $A$ and $B$ consist of three elements each and the fibre over $C$ consists of two elements; the function $f^*:\clE_B\to \clE_A$ induced by $f$ sends $B_0$ to $A_0$ and $B_1,B_2$ both to $A_2$.}
\label{fig_fibres}
\end{figure}
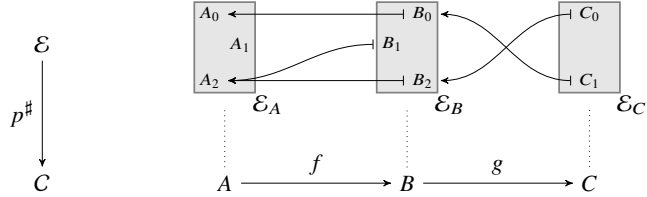

Now we clarify how fibrations relate to the considerations we made back in~\autoref{sec:whatmeaningcannotbe}. It is clear to every one acquainted with categorical approaches to language how~\autoref{def:speaker} was too general to be workable. The notion of fibration seems to fix exactly the kind of issue we would have faced staying in the generality of \autoref{sec:whatmeaningcannotbe}. What seems to be just a happy coincidence ends up being utterly justified by the following result:
\begin{theorem}[Dual of \protect{\cite[Theorem 3]{bams1183534973}}]\label{thm:_factorization}
  Any functor $p : \clD^p \to \clL$ can be written as a composition of functors $\clD^p \xrightarrow{s} \clE^p \xrightarrow{\framew{p}} \clL$, such that $\framew{p}$ is a fibration.
\end{theorem}

This result that we recall in \autoref{stree_fattorizia} below has a direct interpretation in our framework. First, it informs us that we are 'losing nothing' in considering fibrations: as any functor can be factored canonically through a fibration, we do not need any further assumptions to inject fibrations into what has been up to now, a very general picture. We immediately take advantage of this by casting the following definition:
\begin{definition}[Speakers, recast]\label{def: language acquisition}
  In light of~\autoref{thm:_factorization}, given a speaker $\clD^p \xrightarrow{p} \clL$ factorizing as $\clD^p \xrightarrow{s} \clE^p \xrightarrow{\framew{p}} \clL$, we will often abuse notation and refer to $\framew{p}$ as `the speaker $p$'. In short, from now on we postulate that speakers are fibrations.
\end{definition}
Having reassured ourselves how considering fibrations does not lead to any loss of generality, let us look at what~\autoref{thm:_factorization} means in the context of language.
The functor $s$ is projecting the `black box category' $\clD^p$ where meanings live to a category $\clE^p$. We interpret $\clE^p$ as the category, fibred over $\clL$, where concepts have been arranged fully compositionally over the underlying language $\clL$. All the possible meanings of an element of the language $L \in \clL$ constitute the fibre over it.

\begin{equation*}
  \begin{tikzcd}[ampersand replacement=\&]
    \text{We know nothing about this}\ar[r, dashrightarrow] \& \clD^p
    \arrow{dr}{p}
    \arrow{r}{s}
    \&
    \clE^p
    \arrow{d}{\framew{p} }
    \&
    \text{This is compatible with $\clL$}\ar[l, dashrightarrow]  \\
    \& {}
    \& \clL
    \& \text{This can be described}
    \ar[l, dashrightarrow]\\
  \end{tikzcd}
\end{equation*}
The procedure outlined in \cite{bams1183534973} builds a fibration out of $p : \clD^p \to \clL$, in a canonical way; this has to be interpreted as follows: there is a procedure (an algorithm, a construction, a recipe) that a speaker applies to the mental image $\clD^p$ they have built upon $\clL$, that makes $\clD^p$ fit to the language $\clL$ (preserving as much structure as the speaker can), and ready to be shared with other speakers in a meaningful way.
\begin{example}
  Going back to our example sentence \emph{`the cat sleeps'}, in~\autoref{fig: the cat sleeps fibrational} we see how its meaning, which is not necessarily compositional over the sentence, factorises into a string of concepts that are compositional over the words forming the sentence itself. Moreover, given how general~\autoref{thm:_factorization} is, we see that such `compositionalization' of concepts works out for whatever structure we assume on the underlying language.
\end{example}

\begin{figure}[htp]
  \begin{center}
    \begin{tikzpicture}[>=stealth', scale=.725]
      \begin{scope}[xshift=5mm, yshift=0.5cm]
        \node[gray!50] (D) at (0,2) {\color{gray!50}$\clD^p$};
        \node[gray!50] (E) at (2,2) {\color{gray!50}$\clE^p$};
        \node[gray!50] (L) at (2,0) {\color{gray!50}$\clL$};
      \end{scope}
      \node (bettinoDorme) at (-2,4) {\includegraphics[scale=0.05]{figures/bettinoDorme.jpeg}};
      \node (catMeaning) at (4.5,4) {\includegraphics[scale=0.04]{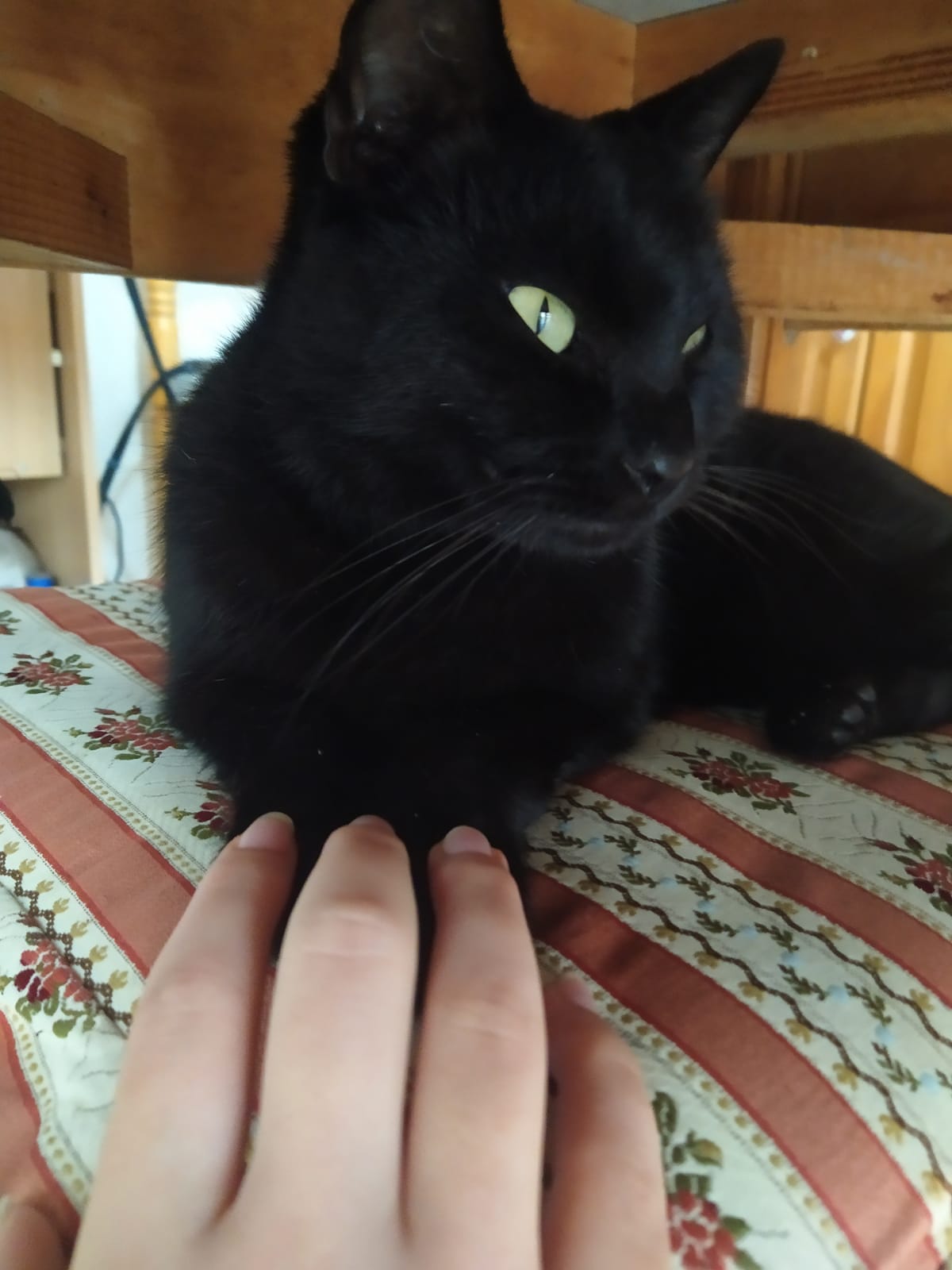}};
      \node (sleepsMeaning) at (7.5,4) {\includegraphics[scale=0.65]{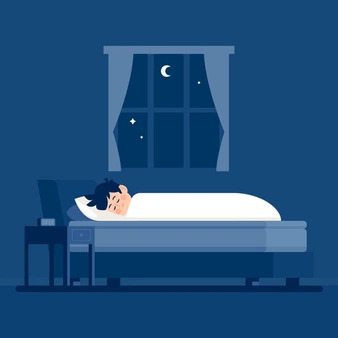}};

      \node (thecat) at (4.5,0){\emph{`the cat}};
      \node (sleeps) at (7.5,-.05){\emph{sleeps'}};
      \draw[->, gray!50] (D) to node[above, font=\footnotesize]{$s$} (E);
      \draw[->, gray!50] (D) to node[below left, font=\footnotesize]{$p$} (L);
      \draw[->, gray!50] (E) to node[right, font=\footnotesize]{$\framew{p}$} (L);

      \draw[|->] (bettinoDorme) -- (catMeaning);

      \draw[|->] (catMeaning) -- (thecat);
      \draw[|->] (sleepsMeaning) -- (sleeps);

    \end{tikzpicture}
  \end{center}
  \caption{Fibrations organize meanings compositionally over language.}\label{fig: the cat sleeps fibrational}
\end{figure}
\section{Fibrations bear gifts}
Thinking fibrationally gave us far more conceptual insights than we may originally have expected. Indeed, in the first draft of this paper we assumed \emph{opfibrations},
which can be understood as fibrations whose reindexing functions are covariant with respect to morphisms in the base. This felt wrong for many reasons; for example, it drew a stark contrast with the common approach to \emph{categorical logic}, which has greatly inspired us from the start. As it often happens in category theory, following any insights given by the mathematical model one is drafting is more rewarding than trying to bend the model to one's will.

To fully explain our model's philosophical insights, we will have to make a short digression. Consider the meaning of the word \emph{`cat'}. One may say that the pictures in~\autoref{fig: the cat sleeps} could represent this word. Indeed, the idea of depicting the meaning in pictorial terms~\cite{Coecke2021} works very well for nouns. However, it is difficult to generalise this insight to verbs: if we ask a hypothetical speaker to `imagine a cat', the speaker may promptly evoke a mental image of a cat. However, can we play the same mental experiment by asking a speaker to depict the meaning of `swimming'? We face the evident problem that `swimming' is a highly contextual linguistic constituent that needs to be paired with a subject to work properly. This apparent difference in behaviour can be resolved if we realise how \emph{the speaker is not evoking the meaning of single, atomic constituents of language, but the meaning of sentences}. In this interpretation, pictures in~\autoref{fig: the cat sleeps} represent the meaning of the sentence \emph{`the cat sleeps'}, together with the meaning of many other sentences that appropriately describe the pictures considered. When asked to imagine a cat, the speaker can project the \emph{cat} out of these pictures, going from the meaning of a fully specified sentence to one of its constituents.

Similarly, the meaning of `swimming' can be represented in the same way by evoking, for instance, a mental image of `someone that swims', and then projecting out the dynamic part of this sentence. This point of view turns the algebraic approach to grammar upside-down: we are not building meanings from atomic constituents but decomposing fully formed meanings into atomic abstractions that depend on context to exist properly. Many recent results in machine learning indeed confirm this idea: algorithms can translate images to text~\cite{Wang2020}, and text to images~\cite{Ramesh2021}, and isolate or blend elements in an image~\cite{Denton2016} -- such as taking the cat out of a picture of a cat and replacing it with another cat. In doing so, the role of sentences is central, as isolated linguistic constituents such as words and particles carry too little context with them to be successfully employable in these tasks.

These considerations brought us to believe that \emph{grammar is algebraic, meaning is coalgebraic}: the former obtains sentences by composing together atomic linguistic constituents, whereas the latter projects and decomposes the meaning of sentences to the meaning of its atomic constituents. Having realised this, it is now apparent that the definition of a fibration, which induces maps between the fibres varying \emph{contravariantly} with respect to the language, is the most appropriate one: language is but a tool guiding us in correctly decomposing mental pictures. We deem it important to stress how all these conceptual realisations stemmed directly from our endeavour to make our model mathematically more natural, not the other way around.
\color{black}
\begin{example}\label{ex: pregroup example}
  If we follow the mantra of DisCoCat~\cite{228d9e4b69926594fd26080f4cfaa9ecfca44eb3,Dostal2016,bc96de1cc022a0425e9f4f607c8d95064a6b3811} and postulate that $\clL$ is a \emph{pregroup}, that is, a thin compact closed category freely generated over a vocabulary~\cite{lambek2010compact,Kartsaklis2016},
  then there is a morphism
  \begin{equation}\label{eq: pregroup example}
    (\emph{the cat}, n) \otimes (\emph{sleeps}, n^l \cdot s) \xrightarrow{f} (\emph{the cat sleeps}, s)
  \end{equation}
  where $n$ denotes the noun type and $s$ the sentence type. For this example, we can define a toy semantics by defining the fibre over every word as the set of sentences that employ that word. So, for instance, the fibre over $(\emph{the cat}, n)$ contains all the possible sentences employing \emph{`the cat'} (which may also include categories). An analogous discourse holds for $(\emph{sleeps}, n^l \cdot s)$,
  while the fibre over $(\emph{the cat sleeps}, s)$ is a singleton. The `meaning' of the morphism $f$ is then given by a section that pics the same element in every fibre: the sentence \emph{the cat sleeps}.

  This example is a toy model, but we deem it very insightful as it constitutes a first, shy step towards the idea of doing distributional models of meaning by assuming distributional statistics of sentences instead of distributional statistics of words. Furthermore, we were happy to discover that similar approaches are currently being investigated by other research groups~\cite{Coecke2021}.
\end{example}
\begin{example}\label{ex: pregroup example 2}
  Let us now build up from the previous example, while also keeping in mind~\autoref{fig_fibres}. Let us consider the fibres over
  elements $(\emph{the cat}, n) \otimes (\emph{sleeps}, n^l \cdot s)$ and $(\emph{the cat sleeps}, s)$, respectively. Furthermore, let us suppose for the sake of simplicity that our fibration $\var[\framew{p}]{\clD^p}{\clL}$, whatever $\clD^p$ is, is furthermore strict monoidal, meaning that
  \begin{equation*}
    \clD^p_{(\emph{the cat}, n) \otimes (\emph{sleeps}, n^l \cdot s)} = \clD^p_{(\emph{the cat}, n)} \times \clD^p_{(\emph{sleeps}, n^l \cdot s)}.
  \end{equation*}
  An element in $\clD^p_{(\emph{the cat sleeps}, s)}$ represents a meaning of the sentence \emph{the cat sleeps}. Now let's focus on the reduction $f$ of~\autoref{eq: pregroup example} and on its corresponding reindexing function:
  \begin{equation*}
    \frac{(\emph{the cat}, n) \otimes (\emph{sleeps}, n^l \cdot s) \xrightarrow{f} (\emph{the cat sleeps}, s)}{\clD^p_{(\emph{the cat sleeps}, s)} \to \clD^p_{(\emph{the cat}, n)} \times \clD^p_{(\emph{sleeps}, n^l \cdot s)}}
  \end{equation*}
  The reindexing is mapping meanings of \emph{the cat sleeps} to meanings of its constituents. As such, we see very explicitly how language can be seen as a blueprint to decompose the meaning of sentences: the sentence structure, going \emph{forward} from components to complete sentences via reduction, informs us how to go \emph{backwards} from the meaning of the sentence to the meaning of its constituents.

  Interestingly, we notice that whereas the meaning of a sentence can be mapped \emph{exactly to one} pair of meanings of its constituents ($f$ induces a function), the opposite is not true: The same pair $(a,b)$ in $\clD^p_{(\emph{the cat}, n)} \times \clD^p_{(\emph{sleeps}, n^l \cdot s)}$ could correspond to \emph{many} different meanings in $\clD^p_{(\emph{the cat sleeps}, s)}$, as visually suggested in~\autoref{fig_fibres}. This suggests that:
  \begin{itemize}
    \item In \fiblang, meaning is \emph{truly} coalgebraic, in that it provides a well-defined way to go from the meaning of a concept to the meanings of it constituents, not the other way around;
    \item The process of 'merging meanings' is inherently ambiguous: thinking about a cat and about a bed is not enough to form the thought of a cat sleeping, as the latter contains many more details that may not be present in the components. As such, merging thoughts really is a creative process.
  \end{itemize}
\end{example}
As we hinted above, the fibrational description of languages establishes almost by default links to what is perhaps one of the most illustrious uses of category theory in logic~\cite{benabou1985fibered,180102927}:
\begin{remark}\label{rem: categorical logic}
  Once a \emph{signature} in the sense of \cite[§1.6]{CLTT} $\Sigma$ is specified, we can build contexts as strings of declarations $\Gamma = (x_1 : \sigma_1,\dots, x_n : \sigma_n)$, and assess judgments like
  \[\Gamma \vdash X : \tau\label{valid}\]
  to express that in context $\Gamma$, a term $X$ has type $\tau\in\Sigma$.

  To every signature, one can associate a \emph{classifying category} $\clC(\Sigma)$, whose objects are contexts $\Gamma$ above, and whose morphisms are suitable \emph{substitutions} of terms one inside the other.

  Now, there is an equivalence between
  \begin{itemize}
    \item models for the theory that the signature prescribes;
    \item functors $\clC(\Sigma) \to \Set$ that preserve finite products;
    \item certain fibrations $\var[\framew{p}]{\clE}{\clC(\Sigma)}$ over $\clC(\Sigma)$.
  \end{itemize}

  Such a fibration in the last item has as fibre over a given $\Gamma\in\clC(\Sigma)$ precisely the category/poset of judgments $X : \tau$ that are valid in context $\Gamma$, i.e. all the judgments like \eqref{valid}.
\end{remark}
In conclusion, a good parallel with our model for language representation --although a naive one for the reader versed in categorical logic-- is that as much as a fibre of $\var{\clE}{\clC(\Sigma)}$ over $\Gamma$ is the set of `judgments that can be deemed true' in context $\Gamma$, a fibre of $\var{\clE}{\clL}$ over $L$ is the set of meanings that can be attributed to $L$ by $\framew{p}$.

This is not all, as the fibrational approach keeps on giving: \autoref{rem: categorical logic} allows us to formulate another example to add to our list in \autoref{roundup_1}-\autoref{roundup_2}.
\begin{example}\label{roundup_3}
    The linguistic doctrine of \emph{maximally connected groupoids}\footnote{
    A maximally connected groupoid (MCG for short) is a category $\clG$ with the property that there exists precisely one morphism between any two given objects; this uniqueness implies that every morphism $a\to b$ is invertible, having as inverse the unique morphism $b\to a$; an MCG is `completely isotropic', in that no object of it can be distinguished from any other object. In this sense, an MCG is a structure that does not carry more information than the mere set of its objects.

    Figure \ref{maximal_gpds} depicts the first few maximally connected groupoids on sets with $1,2,3,4$ elements.
    }
    can be a tentative model for chaotic structures where `all words mean the same'; experimental artlangs as \emph{Zaum} \cite{kruchenykh1913slovo,zaum1,zaum2} (the `transrational' or `beyonsensical' language of the Russian avant-garde movement of early \nth{20} century), Marinetti's \emph{Paroliberismo} \cite{marinetti1972technical,austin1989anna} and D. Stratos' vocal experimentalism \cite{stratos}, as well as C. Vander's \emph{Kobaïan} \cite{magma1,magma2} constitute examples of such situations where the lack of definite meaning for words allows for fuller expression based on `pure sound' experience.

    Given an object $\clL\in\cate{MCG}$, it is of the form $\clG A$ for the functor $\clG$ of \cite[§3]{GRAY1980127} and some set $A$; thus, a functor $p : \clE \to \clL$ is uniquely determined by its function on objects $p_0 : \clE_0\to \clL_0$, and using the fibration property of \autoref{discfib} (and in particular the fact that a cartesian lift of an isomorphism is also an isomorphism) one can prove that
    \begin{itemize}
      \item all fibres have the same cardinality (so, we fix a set $X$ of such cardinality);
      \item a fibration $\var[\framew{p}]{\clE}{\clG A}$ is isomorphic to the projection $\pi'' : X\times\clG A \to \clG A$, in the sense that there exists an isomorphism $H$ of categories over $\clG A$,
      \[\vcenter{\xymatrix{
          \clE \ar[dr]_{\framew{p}} \ar[rr]^H && X\times\clG A \ar[dl]^{\pi''}\\
          & \clG A
      }}\]
    \end{itemize}
    Conceptually, this captures the idea that in these models `all words mean the same' -- since all words are isomorphic in $\clL$ and all fibres have the same cardinality.
    As we summarized above, in these models meaning is conveyed by relying on sound experience and free association of words more than relying on structured grammar.
\end{example}
\begin{figure}[t]
    \begin{center}
    \begin{tikzpicture}[>=stealth', scale=1.5]
        \fill (0,0) circle (1pt);
        \begin{scope}[xshift=1cm]
            \fill (0,0) circle (1pt);
            \fill (1,0) circle (1pt);
            \doublear{(0,0)}{(1,0)}
        \end{scope}
        \begin{scope}[xshift=3.5cm]
            \foreach \i in {0,1,2}
                \fill (\i * 360/3:.5) circle (1pt);
            \doublear{(0:.5)}{(360/3:.5)}
            \doublear{(0:.5)}{(2*360/3:.5)}
            \doublear{(360/3:.5)}{(2*360/3:.5)}
        \end{scope}
        \begin{scope}[xshift=6cm]
            \foreach \i in {0,1,2,3}
                \fill (\i * 360/4:.75) circle (1pt);
            \doublear{(0:.75)}{(360/4:.75)}
            \doublear{(0:.75)}{(2*360/4:.75)}
            \doublear{(0:.75)}{(3*360/4:.75)}
            \doublear{(360/4:.75)}{(2*360/4:.75)}
            \doublear{(3*360/4:.75)}{(2*360/4:.75)}
            \doublear{(360/4:.75)}{(3*360/4:.75)}
        \end{scope}
    \end{tikzpicture}
    \end{center}
    \caption{Maximally connected groupoids of cardinality $n=1,2,3,4$; the category structure of $\clG A$ is uniquely determined by the set of objects $A$, and sending $A$ to $\clG A$ defines a functor $\clG : \Set \to \Cat$ (cf. \cite[§3]{GRAY1980127}). All objects in a MCG are isomorphic in a unique way.}
    \label{maximal_gpds}
\end{figure}
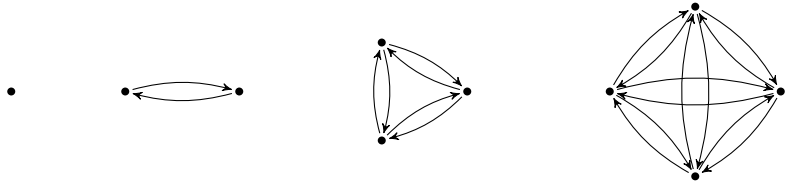

To conclude this section, we set to reconcile our differences between \fiblang and other categorical models of meaning. As we stressed in~\autoref{rem: discocat is upside down},
\fiblang is somewhat `upside-down' in that we are considering functors from meaning to grammar, whereas many other approaches do the opposite. The following theorem allows us to recast \fiblang in more familiar terms, although conceptually less insightful:
\begin{theorem}\label{thm: fibrations functors equivalence}
  Observe that there is a category $\cate{DFib}/\clL$ of fibrations over a given $\clL$, where a map $h : \var[\framew{p}]{\clE^p}{\clL} \to \var[\framew{q}]{\clE^q}{\clL}$ is a functor $h : \clE^p \to \clE^q$ such that $\framew{q}\cdot h=\framew{p}$ (see \autoref{cat_of_fibs}). There is an equivalence of categories:
  \[
    \nabla{} : \cate{DFib}/\clL \cong [\clL^\op, \Set] : \elts{}{}
  \]
  where at the left-hand side, we have the category of all functors $\clL^\op\to\Set$ and natural transformations thereof. The functor $F\mapsto\elts{}{F}$ is often called \emph{the category of elements construction}, or in its most general form \emph{the Grothendieck constuction}. We revise this construction in Appendix A.
\end{theorem}

\begin{example}[Connecting DisCoCat and \fiblang]
  \autoref{thm: fibrations functors equivalence} sheds light on the connection between \fiblang and DisCoCat: let $\clL$ be a Lambek pregroup and consider the functor $F: \clL \to \Set$ sending any object $(\emph{string}, \texttt{type})$ of $\clL$ to the underlying set of the one-dimensional subspace $\langle\emph{string}\rangle$ in the ambient space $\texttt{type}$. Applying $\elts{}{-}$ we obtain a fibration $\elts{}{F}$ where the meaning over each object is the \emph{distributional meaning} of that object as prescribed by DisCoCat.

  On the other hand, applying $\nabla{-}$ to our fibration $\framew{p}$, we get a functor $\nabla{\framew{p}}: \clL \to \Set$ that sends every object $(\emph{string}, \texttt{type})$ to the set of meanings that the string can have, in a way that is compatible with the morphisms of $\clL$. This recovers the usual correspondence from grammar to semantics, as seen in other categorical approaches.
\end{example}
\section{Formal foundations for fibrational linguistics}\label{ct_for_real}
Having explained our point of view in the friendliest way we could and feeling to have conceptually motivated our definitions, we conclude this work by describing \fiblang from a more mathematical standpoint.

Our main tool is the notion of fibration as defined in \autoref{def_ib}, and a number of well-known categorical constructions (mainly, comma objects \cite[1.6]{Bor1} and pullbacks \cite[2.5]{Bor1}, given that (cf. \cite[1.5.1]{CLTT}) the class of fibrations is closed under comma/pullbacks in $\Cat$).

Borrowing from a common terminology to refer to interesting subcategories of $\Cat$ (cf. \cite{9ded893391c4cbe24f74e46bcb6f326826e04049, GUcentazzo}), we give the following definition:
\begin{definition}[Linguistic doctrine]\label{ling_doctrina}
  We call a \emph{(linguistic) doctrine} a subcategory $\cate{LD}\subseteq\Cat$.
\end{definition}
Apart from the usual notational conventions, we
\begin{itemize}
  \item denote with letters $\clL, \clL'$ and the like the objects of a linguistic doctrine of sorts; the generic object $\clL\in\cate{LD}$ will be called a \emph{language};
  \item denote as $U : \cate{LD} \to \Cat$ the forgetful functor discarding the information that $\clL$ is an object of $\cate{LD}$. So, writing $U\clL$ is merely a way to refer to $\clL\in\cate{LD}$ just as a bare category.
\end{itemize}
\begin{definition}[The category $\Spk(\clL)$]\label{speakers}
  In the above notation, let $\clL$ be a language, $\cate{1}$ the terminal category (having a single object and a single identity morphism) and let $\lceil \clL\rceil : \cate{1} \to \Cat$ the functor selecting the object $U\clL \in \Cat$.

  Consider the comma category at the upper left corner of the square
  \[\vcenter{\xymatrix{
        U/\clL \ar[r]\ar[d] \drtwocell<\omit>{} & \cate{LD} \ar[d]^U \\
        1 \ar[r]_{\lceil \clL\rceil} & \Cat.
      }}\]
  The category $\Spk(\clL)$ is defined as the full subcategory of $U/\clL$ spanned by the fibrations $\var[p]{\clE}{\clL}$. In simple words, $\Spk(\clL)$ has
  \begin{itemize}
    \item as objects, the fibrations $\var[p]{\clE}{\clL}$ in $\Cat$;
    \item as morphisms between $\var[p]{\clE}{\clL}$ and $\var[q]{\clF}{\clL}$ all functors $H : \clE \to \clF$ such that $q\cdot H=p$.
  \end{itemize}
\end{definition}
Observe that a morphism of fibrations $H : \var[p]{\clE}{\clL} \to \var[q]{\clF}{\clL}$ commutes with the reindexing operations of \autoref{notat_for_fib}, inducing functors $H_L : \clE_L \to \clF_L$ between each two fibres that fits into commutative squares
\[\vcenter{\xymatrix{
  \clE_L\ar[r]^{H_L}\ar[d]_{u^*} & \clF_L \ar[d]^{u^*} \\
  \clE_{L'} \ar[r]_{H_{L'}} & \clF_{L'}.
  }}\]
  This means (as it should) as morphisms between $\var[p]{\clE}{\clL}$ and $\var[q]{\clF}{\clL}$ translates into a natural transformation between the associated functors $\clL^\op\to\Set$.
\begin{definition}[The category of extended languages]
  The assignment $\clL \mapsto \Spk(\clL)$ can be seen as a functor $\cate{LD}^\op \to \Cat$, acting on morphisms as pullback, from which using the Grothendieck construction of \autoref{groco} we can obtain a fibration $\var[\Lambda]{\cate{ELang}}{\cate{LD}}$ that we dub the category of \emph{extended languages}. A typical object in the total category of $\Lambda$ is a pair $(\clL, p\in\Spk(\clL))$, and a typical morphism $(\clL, p) \to (\clL', q)$ consists of a functor $\clL \to \clL'$ such that $H\cdot p = q$.
\end{definition}
\begin{remark}
  The category of extended languages is conceptually quite important (it collects in a single environment all speakers of all languages), but it has very little practical interest because it contains too many objects and too few morphisms. In fact, by construction, there can be a morphism between objects $p,q$ \emph{only if their domains coincide}. However, this is quite restrictive and unnatural as a request on speakers because it entails that $p,q$ have built the same mental image of their possibly different languages: communication is only possible between speakers of different languages that share a common worldview.

  If anything, we are interested in the opposite situation, namely modelling the process of
  \begin{itemize}
    \item comparing worldviews of speakers of a fixed common language $\clL$, and even better,
    \item let speakers of very different languages, possibly having very different total categories, interact together.
  \end{itemize}
\end{remark}
To this end, we shall consider a `compatible' arrow category on the objects of $\Spk(\clL)$; objects here are the morphisms $\var[p]{\clE}{\clL}$, and a morphism between $p$ and $q$ consists of a commutative square
\[\label{qui}
  \vcenter{\xymatrix{
      \clE \ar[r]^H\ar[d]_p & \clF \ar[d]^q \\
      U\clL \ar[r]_{UF} & U\clL'.
    }}
\]
Here, $F$ must lie in the image of $U$, i.e., it arises from a functor that is a homomorphism for the doctrine of definition. The rationale is that we want to compare $\clL,\clL'$ just as categories, but we want to consider morphisms that are structure-preserving for the defining properties of the doctrine $\cate{LD}$.
\begin{definition}[The category $\Spk^\to$]
  Let $\cate{LD}$ be a linguistic doctrine, and $U : \cate{LD} \to \Cat$ its forgetful functor. The category $\Spk^\to$ is obtained from the comma category $(\Cat/U)$ appearing as left upper corner of the following comma square
  \[\vcenter{\xymatrix{
        (\Cat/U)\ar[r]\ar[d]\drtwocell<\omit>{} & \Cat \ar@{=}[d]\\
        \textsf{LD} \ar[r]_U & \Cat
      }}\]
\end{definition}
as the full subcategory of $(\Cat/U)$ spanned by fibrations $p : \clE \to U\clL$.

Unwinding the definition, the category $\Spk^\to$ has
\begin{itemize}
  \item objects the fibrations $\var[p]{\clE}{U\clL}$, $\var[q]{\clF}{U\clL'}$;
  \item morphisms from $p$ to $q$ the commutative squares as in \eqref{qui}.
\end{itemize}
The category $\Spk^\to$ represents the formal environment underlying \fiblang. As we strive for friendliness and intercommunication with researchers in other fields such as pure linguistics, $\Spk^\to$ will be often obscured or just referenced in upcoming works such as~~\cite{fiblang1}. Nevertheless, we consider it to be of great value, as it is here that our mathematical intuition takes shape.
\section{Conclusion and future work}
This work focused on giving and conceptually motivating the core of \fiblang. In upcoming work \cite{fiblang1} we will focus on applications such as vocabulary acquisition and communication.

Another clear direction of future work is dropping the requisite of our fibrations to be discrete, as allowing more structure in the fibres over words may help capture interesting phenomena. In this sense, the comprehensive factorization of \cite{bams1183534973} can be generalized without pain to the factorization of a \emph{2-functor} $F$ between 2-categories along a canonically chosen non-discrete fibration.

\printbibliography
\appendix
\section{Category theory paraphernalia}
It is an old idea dating back to Grothendieck \cite{Grothendieck1972a,Vistoli2005} that the correspondence between sheaves and local homeomorphisms on a space $X$ can be generalized to non-thin categories. In fact, \emph{every} small category $\clC$ shall be thought of as some sort of `generalized space' \cite{Joyal1991,Hirschowitz2001}; in this perspective, every functor $F : \clC^\op \to \Cat$ -not only those defined over a category of open subsets- shall be thought as a certain generalized fibre bundle over $\clC$, whose fibres are exactly the categories~$FC$.

The so-called \emph{Grothendieck construction} substantiates this idea into a precise theorem, and the theory of \emph{fibrations} provides an analogue for the notion of a space/category `lying over' another and for the notion of local homeomorphism.

In the following, we let $\clC$ be a small category; a thorough presentation of the theory of fibrations is the scope of \cite[Ch. 1]{CLTT} and \cite[B1]{Johnstone2002}; our aim is just to make the axiomatics in the following subsection slightly more self-contained.
In all that follows, $\clE,\clC$ are categories nd $p : \clE \to \clC$ is a functor.
\begin{definition}[Fibration]\label{def_ib}
  The functor $p$ is called a \emph{fibration} if 
    the functor
    \[\vcenter{\xymatrix@R=0mm{\clE^2 \ar[r] & p/\clC \\
      \var{E}{E'} \ar@{|->}[r] & \var{pE}{pE'}}}\]
    induced by the universal property of the comma category $p/\clC$ from the cell
    \[\vcenter{\xymatrix{
      \clE^2 \drtwocell<\omit>{}\ar[r]^s \ar[d]_{p\cdot t} & \clE\ar[d]^p \\
      \clC \ar@{=}[r]& \clC
    }}\]
    has a right adjoint right inverse, i.e. a right adjoint $r : \clC/pE \to \clE/E$ with invertible counit ($s,t : \clE^2\to \clE$ are the source and target functors).\footnote{Strictly speaking, this is the definition of a \emph{cloven} or \emph{split} fibration; but we are not interested in marking the difference between the more general concept of a fibration and a split one, since it would ultimately only obfuscate the intuition. For examples of fibrations that are not cloven, see \cite[p. 10]{180102927}.}
\end{definition}
The domain of a fibration $p$ is often called the \emph{total category}, and its codomain the \emph{base} of the fibration.

Unwinding \autoref{def_ib}, one can see that a functor $p$ is a fibration if and only if the following condition is satisfied:
\begin{quote}
  Every arrow $C\to pE$ in the base has a canonical choice for a \emph{cartesian lift}.
\end{quote}
A \emph{($p$)-cartesian lift} (called a \emph{prone} morphism in \cite{Johnstone2002}) for $u : C\to C'$ consists of a morphism $f : E\to E'$ with the following property: $pf=u$ and for every object $Z\in\clE$, and every pair of morphisms $g : Z\to E'$ above and $w : pZ \to pE=C$ below, arranged in the following manner,
    \[
      \begin{tikzpicture}[>=stealth',baseline=(current bounding box.center)]
        \node[draw] (inE) at (0,3.5) {$\xymatrix@C=4mm@R=4mm{
          &Z\ar[dr]^g \ar@{.>}[dl]_h & \\
          E \ar[rr]_f&&E'
          }$};
        \node[draw] (inB) {$\xymatrix@C=4mm@R=4mm{
          &pZ\ar[dr]^{pg}\ar[dl]_w & \\
          pE \ar[rr]_{pf}&&pE'
          }$};
        \draw[|->] ($(inE.south)!.1!(inB.north)$) -- ($(inE.south)!.9!(inB.north)$);
        \node[right=2mm of inE] {$\clE$};
        \node[right=2mm of inB] {$\clC$};
      \end{tikzpicture}
    \]
there exists a unique $h : Z\to E$ filling the triangle in $\clE$ and mapping it to the triangle in $\clC$ along $p$.

A lot of properties can be deduced from the fact that a certain functor is a fibration: the one that is of interest for us is that such a cartesian lift for a given $C\to pE$ is unique up to a unique isomorphism when it exists.

This ensures that the following fundamental property is true:
\begin{quote}
  Given a morphism $u : C \to C'$ in the base category, there is a functor $u^* : \clE_{C'} \to \clE_C$, where $\clE_A$ is the \emph{fibre} of $p$ over $A$.
\end{quote}
\begin{remark}\label{how_to_save_a_fib}
  The proof that the fibration property entails the existence of reindexings $u^*$ is well-known and can be found in any introductory text on fibration theory: given an object $X\in \clE_{C'}$, the arrow $u : C\to pX$ has a cartesian lift; $u^*X$ is the (uniquely determined) domain of such cartesian lift.
\end{remark}
\begin{remark}
  In general this only defines a \emph{pseudo}functor $\clC^\op \to \Cat$ sending $C$ to $\clE_C$ and $u : C\to C$ to $u^* : \clE_{C'} \to \clE_C$, because although the isomorphisms $u^*\circ v^*\cong (v\circ u)^*$ and $1_C^*\cong 1_{\clE_C}$ are unique, and determined by the universal property of a cartesian lift, they are not the identity in general.
\end{remark}
  \begin{definition}[The category of fibrations]\label{cat_of_fibs}
  The \emph{category of fibrations} over $\clC$ has
  \begin{enumtag}{cf}
    \item objects the fibrations $p : \clE \to \clC$ over $\clC$, regarded as objects of $\Cat/\clC$;
    \item morphisms $(\clE,p) \to (\clF,q)$ the morphisms $h : \clE \to \clF$ of $\Cat/\clC$ that preserve cartesian arrows in the domain category;
    \item 2-cells the natural transformations having `$q$-vertical components', i.e. those $\alpha : h \To h'$ is a natural transformation such that $q * \alpha  =1_p$.
  \end{enumtag}
\end{definition}
An important particular instance of \autoref{def_ib} is when all fibres $\clE_C$ are discrete; in that case, each reindexing $u^* : \clE_{C'}\to \clE_C$ is a \emph{function}, and we obtain a strict functor $\clC^\op \to \Set$.
\begin{definition}[Discrete fibration]\label{discfib}
  A fibration $p : \clE \to \clC$ is called \emph{discrete} if for every morphism $u : C \to pE$ there exists a unique $h : E' \to E$ such that $ph=u$.
\end{definition}
\begin{remark}
  The fibres $\clE_C$ of a functor $p : \clE \to \clC$ play an essential role in the so-called \emph{Grothendieck construction}: every pseudofunctor $\clC^\op \to \Cat$ determines a fibration over $\clC$, its `category of elements'; viceversa, given a functor $p : \clE \to \clC$, the correspondence $C\mapsto \clE_C$ can be `straightened' to a functor $\clC^\op\to\Cat$ if and only if $p$ is a fibration, following the recipe described in \autoref{how_to_save_a_fib}.
\end{remark}
\begin{definition}[Category of elements]\label{eltsf}
  Let $W : \clC^\op \to \Cat$ be a functor; the \emph{category of elements} $\elts{\clC}{W}$ of $W$ is the category having
  \begin{enumtag}{ce}
    \item as objects, the pairs $(C\in\clC, x\in WC)$;
    \item as morphisms $(C,x)\to (C',y)$ those $f\in \clC(C,C')$ such that $W(f)(y)=x$.
  \end{enumtag}
\end{definition}
Evidently, for every $W : \clC^\op \to \Set$, the functor $\var{\elts{\clC}{W}}{\clC}$ sending $(C,x)$ to $C$ and $f$ to $f$ is a discrete fibration, called the \emph{fibration of elements} of $W$.

The Grothendieck construction relies on the fact that \emph{every} fibration $\var[p]{\clE}{\clC}$ is the fibration of elements of a certain functor $W_p : \clC^\op \to \Set$: the functor defined on objects as in \autoref{how_to_save_a_fib}.
\begin{proposition}[The Grothendieck construction]\label{groco}\leavevmode
  Sending a functor $W : \clC^\op\to\Set$ to its category of elements defines a fully faithful functor
  \[\xymatrix{\Cat(\clC^\op,\Set) \,\ar@{^{(}->}[r] & \Cat/\clC;}\]
  whose essential image is the subcategory of discrete fibrations over $\clC$.
\end{proposition}
\begin{construction}[The comprehensive factorization of a functor]\label{stree_fattorizia}
  Given a functor $F : \clC \to \clD$, there exist two factorizations of $F$ into
  \begin{itemize}
    \item a discrete opfibration $p : \clE \to \clD$ followed by an initial functor;
    \item a discrete fibration $q : \clE \to \clD$ followed by a final functor.
  \end{itemize}
  The construction of $p$ given in \cite{bams1183534973} goes as follows. Consider the left Kan extension $k : \clD\to\Set$
  \[\vcenter{\xymatrix{
    \clC \ar[r]^F\ar[d] & \clD\ar[d]^k  \\
    {*} \ar[r] & \Set \ultwocell<\omit>{}
  }}\]
  of the terminal presheaf along $F$; the universal property of the comma object now yields a factorization
  \[\vcenter{\xymatrix{
    \clC \ar[dr]\ar[r] \ar@/^1pc/[rr]^F & (*/k) \ar[r]\ar[d] & \clD\ar[d]^k  \\
    & {*} \ar[r] & \Set \ultwocell<\omit>{}
  }}\]
  where $\clE=(*/k)$ is the category of elements of $k$, and by the Lemma below, a discrete opfibration. The rest of the proof shows that $\clC \to \clE$ is initial, but since we have little interest in the properties of this half of the factorization, we refrain from repeating the argument.
\end{construction}\smallskip
The key ingredient for the proof is the following
\begin{lemma}\label{lemmacomma}
  Given a comma square
  \[\vcenter{\xymatrix{
    (F/G) \drtwocell<\omit>{}\ar[r]\ar[d] & \clA \ar[d]^F\\
    \clB \ar[r]_G & \clC
  }}\]
  \begin{itemize}
  \item the upper horizontal arrow $\var{(F/G)}{\clA}$ is a fibration;
  \item the left vertical arrow $\var{(F/G)}{\clB}$ is an opfibration.
  \end{itemize}
\end{lemma}
\begin{proof}
  See \cite[Exercise 1.4.6]{CLTT}.
\end{proof}
From here, the claim that $p : (*/k) \to\clD$ is a discrete opfibration follows; in order to factor $F$ as a discrete fibration followed by a final functor, just apply the above construction to the opposite of $F$, the functor $F^\op : \clC^\op\to\clD^\op$.
\end{document}